\documentclass[reqno]{amsart}
\usepackage{amsmath,amsthm,amssymb}
\usepackage{latexsym}
\usepackage{eucal}
\usepackage{fullpage}
\usepackage{amsmath}

\newtheorem{theorem}{Theorem}[section]
\newtheorem{lemma}{Lemma}[section]

\def\cal{\mathcal}
\let\Re=\undefined
\DeclareMathOperator{\Re}{Re}
\let\Im=\undefined
\DeclareMathOperator{\Im}{Im}

\def\ge{\geqslant}\def\le{\leqslant}
\def\~{\widetilde}

\begin{document}
\title[ Remark on the formula by Rakhmanov and Steklov's conjecture ]{Remark on the formula by Rakhmanov and Steklov's conjecture }
\author{ S. Denisov }
\address{
\begin{flushleft}
University of Wisconsin--Madison\\  Mathematics Department\\
480 Lincoln Dr., Madison, WI, 53706, USA\\  denissov@math.wisc.edu\\
\end{flushleft}
}\maketitle

\begin{abstract}
The conjecture by Steklov was solved negatively by Rakhmanov in
1979. His original proof was based on the formula for orthogonal
polynomial obtained by adding  point masses to the measure of
orthogonality. In this note, we show how this polynomial can be
obtained by applying the method developed recently for proving the
sharp lower bounds for the problem by Steklov.
\end{abstract}\vspace{1cm}

\section{Introduction: Steklov's conjecture and recent development}

Consider the weight $\rho(x)$ on the interval $[-1,1]$ and the
sequence of polynomials $\{P_{n}(x)\}_{n=0}^{\infty}$, which are
orthonormal
\begin{equation}\label{1s1}
\int^1_{-1}P_n(x)\,P_m(x)\,\rho(x)\,dx=\delta_{n,m}\;,\quad
n,m=0,1,2\,\ldots
\end{equation}
with respect to  $\rho$. Assuming that the leading coefficient of
$P_n(x)$ is positive, these polynomials are defined uniquely.  The
Steklov conjecture dates back to 1921 \cite{steklov} and it asks
{\it whether
 a sequence $\{P_n(x)\}$
is bounded at any point $x_{}{\in}(-1,1)$, provided that $\rho(x)$
is positive on $[-1,1]$, i.e.,
\begin{equation}
\rho(x)\ge \delta, \quad \delta>0\,.
\end{equation}}
This conjecture attracted a lot of attention (check, e.g.,
\cite{Ger1,Ger2,Ger3,Gol} and  a survey \cite{suet}). It was solved
negatively by Rakhmanov in the series of two papers \cite{rakh1,
rakh2}. All existing proofs use the following connection between the
polynomials orthogonal on the segment of the real line and on the
unit circle. Let $\psi, (x\in [-1,1], \psi(-1)=0 )$ be a
non-decreasing bounded function with an infinite number of growth
points. Consider the system of polynomials $\{P_k\}, (k=0,1,\ldots
)$ orthonormal with respect to the measure $d\psi$ supported on the
segment $[-1,1]$. Introduce the function

\begin{equation}
\sigma(\theta)=\left\{
\begin{array}{cc}
-\psi(\cos \theta), & 0\le \theta\le \pi, \\
\psi(\cos \theta), & \pi \le \theta \le 2\pi,%
\end{array}%
\right.\label{trans}
\end{equation}
which is bounded and non-decreasing on $[0,2\pi]$. Consider the polynomials $%
\phi_k(z,\sigma)=\lambda_k z^k+\ldots,$ $\lambda_k>0$ orthonormal
with respect to measure $d\sigma$, i.e.,
\begin{equation}\label{1}
\int^{2\pi}_{0}{\phi_n(e^{i\theta})}\,\overline{\phi_m(e^{i\theta})}\,d\sigma\,=\delta_{n,m}\;,\quad
n,m=0,1,2\,\ldots
\end{equation}
These polynomials can be though of as polynomials orthonormal on the
unit circle $\mathbb{T}$ with respect to a measure $\sigma$ given on
$\mathbb{T}$ as well.

Later, we will use the following notation: for every polynomial
$Q_n(z)=q_nz^n+\ldots+q_0$ of degree at most $n$, we introduce the
$(\ast)$--operation:
\[
Q_n(z)\stackrel{(\ast)}{\longrightarrow} Q_n^*(z)=\bar{q}_0
z^n+\ldots+ \bar{q}_n
\]
This $(\ast)$ depends on $n$. Then, we have the  Lemma.

\begin{lemma}{\rm (\cite{5,6})}\label{reduct}
The polynomial $\phi_n$ is related to $P_k$ by the formula
\begin{equation}\label{real-line}
P_k(x,\psi)=\frac{\phi_{2k}(z,\sigma)+\phi^{*}_{2k}(z,\sigma)}{\sqrt{2\pi
\left[ 1+\lambda_{2k}^{-1}\phi_{2k}(0,\sigma)\right]}}\,
z^{-k},\quad k=0,1,\ldots,
\end{equation}
where $x\!=\!(z\!+\!z^{-1})/2$. \end{lemma} This reduction also
works in the opposite direction: given a measure $\sigma$, defined
on $\mathbb{T}$ and symmetric with respect to $\mathbb{R}$, we can
map it to the measure on the real line and the corresponding
polynomials will be related by \eqref{real-line}.

The version of Steklov's conjecture for the unit circle then reads
as follows:\smallskip

{\flushleft{\it Given $\delta\in (0,1)$ and a probability measure
$\sigma$ which satisfies
\begin{equation}\label{steklo}
\sigma'(\theta)\ge \delta/(2\pi),\quad  {\rm a.e.}\quad \theta\in
[0,2\pi),
\end{equation}
is it true that the sequence $\{\phi_n(z,\sigma)\}$ is bounded for
every $z\in \mathbb{T}$?}}\smallskip

 The normalization
\[
\int d\sigma=1
\]
is not restrictive because of the scaling: $
\phi_n(z,\sigma)=\alpha^{1/2} \phi_n(z,\alpha\sigma), \,\alpha>0$.
The negative answer to this question (see \cite{rakh1}) implied the
solution to Steklov's conjecture on the real line due to Lemma
\ref{reduct}.

 Besides the orthonormal polynomials, we can define the monic
orthogonal ones $\{\Phi_n(z,\sigma)\}$ by requiring
\[
{\rm
coeff}(\Phi_n,n)=1,\,\int_{0}^{2\pi}\Phi_n(e^{i\theta},\sigma)\overline{\Phi_m(e^{i\theta},\sigma)}\,d\sigma\,=0\;,\quad
m<n,
\]
where ${\rm coeff}(Q,j)$ denotes the coefficient in front of $z^j$
in the polynomial $Q$.

 The original argument by Rakhmanov was based on the following
formula for the orthogonal polynomial that one gets after adding
several point masses to a ``background" measure at particular
locations on the circle.
\begin{lemma}{\rm ({\bf Rakhmanov's formula}, \,\cite{rakh1})}\quad\label{geron}Let $\mu$ be a positive measure on $\mathbb T$ with infinitely many growth points and
\[
K_n(\xi,z,\mu)=\sum^n_{j=0}\overline{\phi_j(\xi,\mu)}\phi_j(z,\mu)
\]
be the Christoffel-Darboux kernel, i.e.,
\[
P(\xi)=\langle
P(z),K_n(\xi,z,\mu)\rangle_{L^2(\mathbb{T},\mu)},\quad \forall P:
\,\deg P\le n\,\,.
\]
Then, if $\{\xi_j\}\in\mathbb T,j=1,\,...\,,m,\,m< n$ are chosen
such that
\begin{equation}\label{uslo-k}
K_{n-1}(\xi_j,\xi_l,\mu)=0, \quad j\neq l\end{equation} then
\begin{equation}
\Phi_n(z,\eta)=\Phi_n(z,\mu)-\sum_{k=1}^m
\frac{m_k\Phi_n(\xi_k,\mu)}{1+m_kK_{n-1}(\xi_k,\xi_k,\mu)}K_{n-1}(\xi_k,z,\mu)
\end{equation}
where
\[
\eta=\mu+\sum_{k=1}^m m_k\delta_{\theta_k}, \quad
\xi_k=e^{i\theta_k},\quad m_k\ge 0\,\,.
\]
\end{lemma}
It is known (\cite{rakh1}) that for every $\widehat z\in
\mathbb{T}$, the function $K_{n-1}(\xi,\widehat z,\mu)$ has exactly
$n-1$ different roots $\{\xi_j(\widehat z)\}, j=1,\ldots,n-1$ and
they all lie on $\mathbb{T}$. Moreover, $K_{n-1}(\xi_i,\xi_j,\mu)=0$
for $i\neq j$ (see \cite{rakh1}). The limitation that $\{\xi_j\}$
must be the roots is quite restrictive and the direct application of
this formula with background $d\mu=\delta (2\pi)^{-1}d\theta$ yields
only logarithmic lower bound in the following variational problem:
\begin{equation}\label{loga}
M_{n,\delta}=\sup_{\sigma\in S_\delta}
\|\phi_n(z,\sigma)\|_{L^\infty(\mathbb{T})}\ge C(\delta)\log n,
\quad n>n_0(\delta)
\end{equation}
and $S_\delta$ denotes the class of probability measures that
satisfy \eqref{steklo}. The straightforward iteration of this
``fixed-$n$, varying $\sigma$" construction gave the negative
solution to the original conjecture of Steklov
(\cite{rakh1}).\medskip

{\bf Remark.} It is known \cite{sim1} that for probability measures
$\sigma$ in the  Szeg\H{o} class, i.e., those  $\sigma$ for which
\[
\int_0^{2\pi} \log \sigma' d\theta>-\infty,
\]
 we have
 \[
\exp\left(\frac{1}{4\pi}\int_{\mathbb{T}} \log(2\pi \sigma'(\theta))
d\theta\right)\leq
\left|\frac{\Phi_n(z,\sigma)}{\phi_n(z,\sigma)}\right|\le 1, \quad
\forall z\in \mathbb{C}
\]
Thus, for measures in Steklov class, i.e., those satisfying
\eqref{steklo}, the following estimate holds
\[
\sqrt\delta\leq
\left|\frac{\Phi_n(z,\sigma)}{\phi_n(z,\sigma)}\right|\le 1, \quad
\forall z\in \mathbb{C}
\]
so, is $\Phi_n$ or $\phi_n$ grow in $n$, they grow simultaneously.

The upper bound for $M_{n,\delta}$ is easy to obtain
\begin{equation}\label{rooot}
M_{n,\delta}\le C(\delta)\sqrt n
\end{equation}
and the corresponding result for fixed $\sigma\in S_\delta$ and
$n\to\infty$ is contained in the following Lemma.
\begin{lemma}{\rm (\cite{adt})}\,
If $\sigma\in S_\delta$, then
\begin{equation}\label{steklo2}
\|\phi_n(z,\sigma)\|_{L^\infty(\mathbb{T})}=o(\sqrt n), \quad
n\to\infty\,.
\end{equation}
\end{lemma}
 The gap between $\log n$ and $\sqrt n$ was nearly closed in
the second paper by Rakhmanov \cite{rakh2} where the following bound
was obtained:
\[
M_{n,\delta}\ge C(\delta)\sqrt{\frac{n}{\log^3 n}}
\]
under the assumption that $\delta$ is small.\smallskip

 In the recent paper \cite{adt}, the following two Theorems were
 proved.
\begin{theorem}{\rm (\cite{adt})}\label{T3-i}         If $\delta\in (0,1)$ is fixed, then
\begin{equation}\label{osnova-i}M_{n,\delta} > C(\delta)\sqrt n\,\, .
\end{equation}
\end{theorem}
\smallskip and\smallskip
\begin{theorem}{\rm (\cite{adt})}\label{rrra-i}            Let $\delta{\in}(0,1)$ be fixed. Then, for every positive sequence
$\{\beta_n\}:\lim_{n\to\infty}\beta_n=0$, there is a probability
measure $\sigma^*:d\sigma^*={\sigma^*}'d\theta,\,
\sigma^*{\in}S_\delta$ such that
\begin{equation}\label{est-ra-i}
\|\phi_{k_n}(z,\sigma^*)\|_{L^\infty(\mathbb T)}\ge
\beta_{k_n}\sqrt{k_n}
\end{equation}
for some sequence $\{k_n\}\subset\mathbb N$.\end{theorem} These two
results completely settle the problem by Steklov on the sharpness of
 estimates \eqref{rooot} and \eqref{steklo2}. The method used in
the proof was very different from those of Rakhmanov. In the current
paper, we will show that it can be adjusted to the cover
construction by Rakhmanov. This new modification is interesting in
its own as it contains certain cancelation different from the one
used in \cite{adt}.\smallskip

The structure of the paper is as follows. The second section
contains the explanation of the main idea used in \cite{adt} to
prove Theorem \ref{T3-i}. In the third one, we show how it can be
used to cover the Rakhmanov's construction.

We will use the following notation. The Cauchy kernel $C(z,\xi)$ is
defined as
\[
C(z,\xi)=\frac{\xi+z}{\xi-z},\quad\xi\in\mathbb T\,\,.
\]
The function analytic in $\mathbb{D}=\{z:|z|<1\}$ is called
Caratheodory function if its real part is nonnegative
in~$\mathbb{D}$. Given a set $\Omega$, $\chi_\Omega$ denotes the
characteristic function of $\Omega$. If two positive functions
$f_{1(2)}$ are given, we write $f_1\lesssim f_2$ if there is an
absolute constant $C$ such that
\[
f_1<Cf_2
\]
for all values of the argument. We define $f_1\gtrsim f_2$
similarly. Writing $f_1\sim f_2$  means $f_1\lesssim f_2\lesssim
f_1$.

\bigskip

\section{Method used to prove Theorem \ref{T3-i}}

In this section we explain an idea used in the proof of Theorem
\ref{T3-i}. We start with recalling some basic facts about the
polynomial orthogonal on the unit circle. With any probability
measure $\mu$, which is defined on the unit circle and have
infinitely many growth points, one can associate the orthonormal
polynomials of the first and second kind, $\{\phi_n\}$ and
$\{\psi_n\}$, respectively. $\{\phi_n\}$ satisfy the following
recursions (\cite{sim1}, p. 57) with Schur parameters
$\{\gamma_n\}$:
\begin{equation}\left\{\begin{array}{cc}
\phi_{n+1}=\rho_n^{-1}(z\phi_n-\overline\gamma_n\phi_n^*),&\phi_0=1\\
\phi_{n+1}^*=\rho_n^{-1}(\phi_n^*-\gamma_n z\phi_n),&\phi^*_0=1
\end{array}\right.
\label{srecurs}
\end{equation}
and $\{\psi_n\}$ satisfy the same recursion but with Schur
parameters $\{-\gamma_n\}$, i.e.,
\begin{equation}\label{secon}\left\{\begin{array}{cc}
\psi_{n+1}=\rho_n^{-1}(z\psi_n+\overline\gamma_n\psi_n^*),&\psi_0=1\\
\psi_{n+1}^*=\rho_n^{-1}(\psi_n^*+\gamma_n z\psi_n),&\psi^*_0=1
\end{array}\right.\end{equation}
The coefficient $\rho_n$ is defined as
\[
\rho_n=\sqrt{1-|\gamma_n|^2}
\]
The following Bernstein-Szeg\H{o} approximation  is valid:
\begin{lemma}{\rm (\cite{5},\cite{sim1})}\, Suppose $d\mu$ is a probability measure and $\{\phi_j\}$ and $\{\psi_j\}$ are the
corresponding orthonormal polynomials of the first/second kind,
respectively. Then, for any $N$, the Caratheodory function
\[
F_N(z)=\frac{\psi_N^*(z)}{\phi_N^*(z)}=\int_{\mathbb T}
C(z,e^{i\theta})d\mu_N(\theta),\,\,{\rm where}\quad
d\mu_N(\theta)=\frac{d\theta}{2\pi|\phi_N(e^{i\theta})|^2}=
\frac{d\theta}{2\pi|\phi^*_N(e^{i\theta})|^2}
\]
has the first $N$ Taylor coefficients  identical to the Taylor
coefficients of the function
\[
F(z)=\int_{\mathbb T}C(z,e^{i\theta})d\mu(\theta)\,\,.
\]
In particular, the polynomials $\{\phi_j\}$ and $\{\psi_j\}$,
$j\!\le\!N$ are the orthonormal polynomials of the first/second kind
for the measure $d\mu_N$.
\end{lemma}
We also need the following Lemma which can be verified directly:
\begin{lemma}\label{vspomag}
The polynomial $P_n(z)$ of degree $n$ is the orthonormal polynomial
for a probability measure with infinitely many growth points if and
only if
\begin{itemize}
\item[1.] $P_n(z)$ has all $n$ zeroes inside $\mathbb D$ (counting the multiplicities).
\item[2.] The normalization conditions
\[
\int_\mathbb
T\frac{d\theta}{2\pi|P_n(e^{i\theta})|^2}=1~,\quad\operatorname{coeff}(P_n,n)>0
\]
are satisfied.\end{itemize}\end{lemma}
\begin{proof} Take $2\pi|P_n(e^{i\theta})|^{-2}d\theta$ itself as a
probability measure. The orthogonality is then immediate.
\end{proof}

We continue with a Lemma which paves the way for constructing the
measure giving, in particular, the optimal bound \eqref{osnova-i}.
It is a special case of a solution to the
 truncated moment's problem.

\begin{lemma}\label{decop}
Suppose we are given a polynomial $\phi_n$ and  Caratheodory
function $\~F$
 which satisfy the following properties
\begin{itemize}
\item[1.] $\phi_n^*(z)$ has no roots in $\overline{\mathbb D}$.
\item[2.] Normalization on the size and ``rotation\!"
\begin{equation}\label{norma}
\int_\mathbb T|\phi_n^*(z)|^{-2}d\theta
=2\pi~,\quad\phi_n^*(0)>0\,\,.
\end{equation}

\item[3.] $\~F\!\in\!C^\infty(\mathbb T)$, $\Re\~F>0$ on $\mathbb T$, and
\begin{equation}\label{norka}
\frac1{2\pi}\int_\mathbb T\Re\~F(e^{i\theta})d\theta=1\,\,.
\end{equation}

\end{itemize}
Denote the Schur parameters given by the probability measures
$\mu_n$ and $\~\sigma$
\[
d\mu_n=\frac{d\theta}{2\pi|\phi_n^*(e^{i\theta})|^2}, \quad
d\~\sigma=\~\sigma'd\theta=\frac{\Re \~F(e^{i\theta})}{2\pi}d\theta,
\]
as $\{\gamma_j\}$ and $\{\~\gamma_j\}$, respectively. Then, the
probability measure $\sigma$, corresponding to Schur coefficients
\[
\gamma_0,\ldots, \gamma_{n-1},\~\gamma_0,\~\gamma_1,\ldots
\]
is purely absolutely continuous with the weight given by
\begin{equation}\label{mp}
\sigma'=\frac{4\~\sigma'}{|\phi_n+\phi_n^*+\~F(\phi_n^*-\phi_n)|^2}=\frac{2\Re\~F}
{\pi|\phi_n+\phi_n^*+\~F(\phi_n^*-\phi_n)|^2}\,\,.
\end{equation}
The polynomial $\phi_n$ is the orthonormal polynomial for $\sigma$.
\end{lemma}
The proof of this Lemma is contained in \cite{adt}. We, however,
prefer to give its sketch  here.

\begin{proof}
First, notice that $\{\~\gamma_j\}\in \ell^1$ by Baxter's Theorem
(see, e.g., \cite{sim1}, Vol.1, Chapter 5). Therefore, $\sigma$ is
purely absolutely continuous by the same Baxter's criterion. Define
the orthonormal polynomials of the first/second kind corresponding
to measure $\~\sigma$ by $\{\~\phi_j\}, \{\~\psi_j\}$. Similarly,
let $\{\phi_j\}, \{\psi_j\}$ be orthonormal polynomials for
$\sigma$. Since, by construction, $\mu_n$ and $\sigma$ have
identical first $n$ Schur parameters, $\phi_n$ is $n$-th orthonormal
polynomial for~$\sigma$.

 Let
us compute the polynomials $\phi_j$ and $\psi_j$, orthonormal with
respect to $\sigma$, for the indexes $j>n$. By \eqref{secon}, the
recursion can be rewritten in the following matrix form
\begin{equation}\label{m-ca}
\left(\begin{array}{cc}
\phi_{n+m} & \psi_{n+m}\\
\phi_{n+m}^* & -\psi_{n+m}^*
\end{array}\right)=\left(\begin{array}{cc}
{\cal A}_m & {\cal B}_m\\
{\cal C}_m & {\cal D}_m
\end{array}\right)\left(\begin{array}{cc}
\phi_{n} & \psi_{n}\\
\phi_{n}^* & -\psi_{n}^*
\end{array}\right)\end{equation}
where ${\cal A}_m, {\cal B}_m, {\cal C}_m, {\cal D}_m$ satisfy
\begin{eqnarray*}\left(\begin{array}{cc}
{\cal A}_0 & {\cal B}_0\\
{\cal C}_0 & {\cal D}_0
\end{array}\right)=\left(\begin{array}{cc}
1 & 0\\
0 & 1
\end{array}\right),\hspace{6cm}\\
\left(\begin{array}{cc}
{\cal A}_m & {\cal B}_m\\
{\cal C}_m & {\cal D}_m
\end{array}\right)=\frac1{\~\rho_0\cdot\ldots\cdot\~\rho_{m-1}}\left(\begin{array}{cc}
z & -\~\gamma_{m-1}\\
-z\~\gamma_{m-1} & 1
\end{array}\right)\cdot\ldots\cdot\left(\begin{array}{cc}
z & -\~\gamma_0\\
-z\~\gamma_0 & 1
\end{array}\right)\end{eqnarray*}
and thus depend only on $\~\gamma_0,\ldots,\~\gamma_{m-1}$.
Moreover, we have
\[\left(\begin{array}{cc}
\~\phi_m & \~\psi_m\\
\~\phi_m^* & -\~\psi^*_m
\end{array}\right)=\left(\begin{array}{cc}
{\cal A}_m & {\cal B}_m\\
{\cal C}_m & {\cal D}_m
\end{array}\right)\left(\begin{array}{cc}
1 & 1\\
1 & -1
\end{array}\right)
\,\,.
\]
Thus, ${\cal A}_m\!=\!(\~\phi_m\,{+\,\~\psi_m)/2,~{\cal
B}_m\!=\!(\~\phi_m\,-}\,\~\psi_m)/2,~ {\cal
C}_m\!=\!(\~\phi^*_m\,{-\,\~\psi^*_m)/2,~{\cal
D}_m\!=\!(\~\phi^*_m\,+}\,\~\psi^*_m)/2$ and their substitution into
\eqref{m-ca} yields
\begin{equation}\label{intert}
2\phi_{n+m}^*=\phi_n(\~\phi_m^*-\~\psi^*_m)+\phi_n^*(\~\phi_m^*+\~\psi^*_m)=
\~\phi_m^*\left(\phi_n+\phi_n^*+\~
F_m(\phi_n^*-\phi_n)\right)\end{equation} where
\[
\~F_m(z)=\frac{\~\psi^*_m(z)}{\~\phi^*_m(z)}\,\,.
\]
Since $\{\~\gamma_n\}\!\in\!\ell^1$ and $\{\gamma_n\}\!\in\!\ell^1$,
we have (\cite{sim1}, p.~225)
\[
\~F_m\to\~F~{\rm as~}m\to\infty~{\rm and~}
\phi_j^*\to\Pi,~\~\phi_j^*\to\~\Pi~{\rm as~}j\to\infty\,\,.
\]
uniformly on $\overline{\mathbb D}$. The functions $\Pi$ and $\~\Pi$
are the Szeg\H{o} functions of $\sigma$ and $\~\sigma$,
respectively, i.e., they are the outer functions in $\mathbb D$ that
satisfy
\begin{equation}\label{facti}
|\Pi|^{-2}=2\pi\sigma',\quad |\~\Pi|^{-2}=2\pi\~\sigma'\,\,
\end{equation}
on $\mathbb{T}$. In \eqref{intert}, send $m\to\infty$ to get
\begin{equation}\label{facti1}
2\Pi=\~\Pi\left(\phi_n+\phi_n^*+\~F(\phi_n^*-\phi_n)\right)
\end{equation}
and we have \eqref{mp} after taking the square of absolute values
and using \eqref{facti}.
\end{proof}
In \cite{adt}, to prove \eqref{osnova-i} with small $\delta$, the
polynomial $\phi_n$ and $\~F$ were chosen to satisfy  extra
conditions (see Decoupling Lemma in \cite{adt}):
\begin{equation}\label{od1}
|\phi_n(1)|>C\sqrt n
\end{equation}
and
\begin{equation}\label{sec1}
|\phi_n^*(z)|+|\~ F(z)(\phi_n^*(z)-\phi_n(z))|\le C\sqrt{\Re
\~F(z)},\quad z\in \mathbb{T}
\end{equation}
\eqref{od1} yields the
 $\sqrt n$--growth claimed in Theorem \ref{T3-i}.
 The last inequality guarantees that $\sigma$ belongs to Steklov
class due to \eqref{mp} and \eqref{facti}. However, as will be made
clear in the next section, \eqref{sec1} is not necessary for
polynomials to have large uniform norm.
\medskip

\section{Rakhmanov's construction via new approach}

Our goal in this section is twofold. Firstly, we use the method
explained in section 2 to reproduce Rakhmanov's polynomial and
polynomials with the similar structure that have large uniform norm
and which are orthogonal with respect to a measure in Steklov class.
Secondly, we show that the last condition in the Decoupling Lemma
(\cite{adt}, formula $(3.6)$, or, what is the same, the bound
\eqref{sec1} above) is not really necessary for the orthogonal
polynomial to have large uniform norm. Instead, that can be achieved
by a different sort of cancelation which might be of its own
interest.\smallskip

We start with recalling the construction by Rakhmanov \cite{rakh1}.
In Lemma \ref{geron}, take the Lebesgue measure $\mu:
d\mu=d\theta/(2\pi)$. We have the following expression for the
kernel
\[
K_{n-1}(\xi,z,\mu)=\sum_{j=0}^{n-1}\bar\xi^jz^j=\frac{(z\bar\xi)^n-1}{z\bar\xi-1}
\]
Given two parameters $\epsilon, (0<\epsilon<1)$ and $m, (m<n-1)$, we
add the mass $m_k=\epsilon m^{-1}$ to each of the points
$\xi_k=e^{i2\pi k/n},\, k=0,\ldots, m-1$. Then Lemma \ref{geron}
gives

\[
\Phi_n(z,\mu)=z^n-\frac{\epsilon m^{-1}}{1+\epsilon n m^{-1}}
\sum_{j=0}^{m-1} \Bigl((\overline{\xi}_{j})^{\,n-1}z^{n-1}+
\ldots+\overline{\xi}_jz+1 \Bigr)
\]
and therefore
\begin{equation}\label{leseq}
\Phi_n^*(z,\mu)=1-\frac{\epsilon m^{-1}}{1+\epsilon n m^{-1}}(d_1
z+d_2z^2+\ldots+d_nz^n)
\end{equation}
\[
d_l=\sum_{j=0}^{m-1}\xi_j^{n-l}=\sum_{j=0}^{m-1}\xi_j^{-l}, \quad
l=1,\ldots, n
\]
Thus, if $n$ is even and $m=n/2$, we have
\begin{equation}\label{choisi-7}
d_n=m, \quad d_l=\frac{(-1)^l-1}{e^{-i2\pi l/n}-1},\quad l=1,\ldots,
n-1
\end{equation}
and $ \overline{d}_{n-l}=d_l, \,l=1,\ldots, n-1$. Then,
\begin{equation}\label{curi-curi}
 \Phi_n^*-\Phi_n=\left(\frac{1+3\epsilon}{1+2\epsilon}\right)(1-z^n), \quad \|\Phi_n^*-\Phi_n\|_{L^\infty(\mathbb{T})}<C
\end{equation}
Since
\begin{equation}\label{taytay}
e^{-i2\pi l/n}-1=-i\frac{2\pi l}{n}+O\left(\frac{l^2}{n^2}\right),
\quad l<0.01 n,
\end{equation}
it is clear that $\|\Phi_n\|_{L^\infty(\mathbb{T})}\sim 1+\epsilon
\log n$ and this growth occurs around the points $z=1$ and $z=-1$.
The choice of $\{m_j\}$ can be rather arbitrary and does not have to
be given by equal mass distribution to provide the logarithmic
growth. Since $\|\eta\|=1+\epsilon$ and $\eta'=(2\pi)^{-1}$, the
normalized measure $\eta/\|\eta\|\in S_\delta,
\delta=(1+\epsilon)^{-1}$
\[
\|\phi_n(z,\eta/\|\eta\|)\|_{L^\infty(\mathbb{T})}\sim 1+\epsilon\ln
n
\]
This argument proves \eqref{loga}.

The next theorem is the main result of the paper. It explains how
the polynomial of the structure similar to \eqref{leseq} can be
obtained by the method described in the previous section.

\begin{theorem}\label{lbl}
For every $\epsilon\in (0,1)$, there is $\sigma=\sigma'd\theta$:
\[
\int_{0}^{2\pi} d\sigma=1, \quad
|\sigma'(\theta)-(2\pi)^{-1}|\lesssim  \epsilon
\]
and
\[
\|\phi_n(z,\sigma)\|_{L^\infty(\mathbb{T})}\sim \epsilon \log n
\]
\end{theorem}
\begin{proof}
We will consider an analytic polynomial $M_{n}$ of degree $n-1$
satisfying
 two conditions
\begin{equation}\label{condicion}
\int_{0}^{2\pi} \Re M_n(e^{i\theta})d\theta=0, \quad\|\Re
M_n(e^{i\theta})\|_{L^\infty(\mathbb{T})}<C, \quad \|\Im
M_n(e^{i\theta})\|_{L^\infty(\mathbb{T})}\sim \log n
\end{equation}
This $M_n$ is easy to find. Consider $
l(\theta)=\chi_{0<\theta<\pi}-\chi_{\pi<\theta<2\pi} $ and take
\[
L(z)=\cal{C}(l)=\frac{1}{2\pi}\int_{0}^{2\pi}
C(z,e^{i\theta})l(e^{i\theta})d\theta=\frac{1}{2\pi}\int_{0}^{2\pi}
\frac{e^{i\theta}+z}{e^{i\theta}-z}l(e^{i\theta})d\theta
\]
$\Re {C}$ is the Poisson kernel so $\Re L(e^{i\theta})=l(\theta),\,
\theta\neq 0,\pi$. Then, we take $M_n=\cal{F}_{n}\ast L$, where
$\cal{F}_n$ is the Fejer kernel. Since $\cal{F}_n$ is real,
nonnegative trigonometric polynomial of degree $n-1$ and
$\|\cal{F}_n\|_{L^1[0,2\pi]}=1$, we have
\[
|\Re M_n|=|\cal{F}_n\ast l|\le 1,\,z\in \mathbb{T}; \quad
\int_{0}^{2\pi} \Re M_n(e^{i\theta})d\theta =0, \quad M_n(0)=0
\]
The logarithmic growth of $M_n$ around the points $\theta=0$ and
$\theta=\pi$ is a standard exercise, e.g.,
\begin{equation}\label{sniii}
|\Im M_n(e^{i\theta})|\sim \log n , \quad |\theta|<C n^{-1}
\end{equation}
with arbitrary large fixed $C$. Now, take a small positive $
\epsilon$ and define
\begin{equation}\label{new1-fr}
\widetilde F=1-2\epsilon M_n, \quad D_n=M_n+b,\quad
\phi_n^*=a(1+\epsilon(D_n+D_n^*))
\end{equation}
where $a$ and $b$ are positive parameters to be chosen later so that
all conditions of the Lemma \ref{decop} are satisfied. We have
\[
\phi_n=a(z^n+\epsilon(D_n+D_n^*))
\]
$\Bigl( \Bigr.$Notice that $ \phi_n^*-\phi_n=a(1-z^n) $ and compare
it with \eqref{curi-curi}.$\Bigl. \Bigr)$ Since $D_n(0)=b$ and $\deg
D_n=n-1$, we have $D_n^*(0)=0$ and then $\phi_n^*(0)=a(1+\epsilon
b)>0$. Let us check other normalization conditions for these
functions.
\begin{equation}\label{deviatsia}
\Re \widetilde F(e^{i\theta})=1+O(\epsilon)>0,\quad \int_{0}^{2\pi}
\Re \widetilde F(e^{i\theta})d\theta=2\pi
\end{equation}
Choose $b$  such that $\Re D_n\in [C_1,C_2]$ with $C_1>0$. For
example, if $b=2$, then $\Re D_n\in [1,3]$. We can write
\[
1+\epsilon(D_n+D_n^*)=D_n\Bigl(\epsilon(1+e^{i(n\theta-2\Theta_n)})+D_n^{-1}\Bigr),
\quad z=e^{i\theta}\in \mathbb{T}
\]
where $\Theta_n=\arg D_n$. Notice that $D_n$ is zero free in
$\overline{\mathbb{D}}$ since it has positive real part on
$\mathbb{T}$. Since
\[
\Re \Bigl(1+e^{i(n\theta-2\Theta_n)}\Bigr)\ge 0, \quad \Re
D_n^{-1}=\frac{\Re D_n}{|D_n|^2}>0
\]
we have that $\phi^*_n$ is zero free in $\overline{\mathbb{D}}$.
Then, for $z\in \mathbb{T}$,
\[
|1+\epsilon(D_n+D_n^*)|\ge \frac{|\Re D_n|}{|D_n|}\sim |D_n|^{-1}
\]
and
\begin{equation}\label{ra-no}
\int_{0}^{2\pi} |1+\epsilon(D_n+D_n^*)|^{-2} d\theta\lesssim
\int_{0}^{2\pi} |D_n|^2d\theta\lesssim 1
\end{equation}
since $\|D_n\|_{L^2(\mathbb{T})}\lesssim
b+\|L\|_{L^2(\mathbb{T})}\lesssim 1$. On the other hand,
\[
\|\Bigl(1+\epsilon(D_n+D_n^*)\Bigr)-1\|_{L^2(\mathbb{T})}\leq
2\epsilon\|D_n\|_{L^2(\mathbb{T})}\lesssim \epsilon
\]
and so
$\|1+\epsilon(D_n+D_n^*)\|_{L^2(\mathbb{T})}=\sqrt{2\pi}+O(\epsilon)$.
From Cauchy-Schwarz, we get
\[
2\pi\leq
\|1+\epsilon(D_n+D_n^*)\|_{L^2(\mathbb{T})}\|\bigl(1+\epsilon(D_n+D_n^*)\bigr)^{-1}\|_{L^2(\mathbb{T})}
\]
and so
\[
\frac{2\pi}{\sqrt{2\pi}+O(\epsilon)}\leq
\|\bigl(1+\epsilon(D_n+D_n^*)\bigr)^{-1}\|_{L^2(\mathbb{T})}\lesssim
1
\]
Let us choose $a$ so that
\[
\int_{0}^{2\pi} |\phi_n^*|^{-2}d\theta=2\pi
\]
which implies $a\sim 1$. We satisfied all  conditions of the Lemma
\ref{decop}. Consider the formula  \eqref{mp}. We can write
\begin{equation}\label{ushi}
\phi_n+\phi_n^*+\widetilde F (\phi_n^*-\phi_n)=2\phi_n^*-2\epsilon
M_n(\phi_n^*-\phi_n)=
\end{equation}
\[
2\phi_n^*-2a\epsilon (M_n-M_nz^n)=2a\Bigl((1+\epsilon
(D_n+D_n^*))-\epsilon (M_n-M_nz^n)\Bigr)=
\]
\[
2a\Bigl((1+\epsilon (D_n+D_n^*))-\epsilon (M_n-(M_n+\overline
M_n-\overline M_n)z^n)\Bigr) =2a\Bigl(1+\epsilon b(1+z^n)+2\epsilon
z^n \Re M_n \Bigr)
\]
Let us control the deviation  of  $\sigma'$ from the constant. We
get
\[
2\pi\sigma'=4(2a)^{-2}\cdot \Re \widetilde F \cdot
|1+O(\epsilon)|^{-2}=a^{-2}(1+O(\epsilon))\cdot |1+O(\epsilon)|^{-2}
\]
where we used $|\Re M_n|\le 1$ and \eqref{deviatsia}. Since $a\sim
1$, we have that the deviation of $2\pi\sigma'$ from  $a^{-2}$ is at
most $C\epsilon$. Since $\sigma$ is a probability measure, this
implies $a=1+O(\epsilon)$. We are left to show that
$\|\phi_n\|_{L^\infty(\mathbb{T})}\sim \log n $. By construction, it
is sufficient to prove
\[
\|M_n+z^n\overline M_n\|_{L^\infty(\mathbb{T})}\sim \log n
\]
Indeed,
\[
|M_n(\widetilde z_n)+\widetilde z_n^{\,n}\overline M_n(\widetilde
z_n)|\sim \log n, \quad \widetilde z_n=e^{i\pi/n}
\]
as follows from \eqref{sniii}.
\end{proof}

{\bf Remark.} Our analysis covers the polynomial \eqref{leseq}
constructed by Rakhmanov too. If $d_0=m$, we can rewrite
\eqref{leseq} as
\begin{equation}\label{new-fr}
\Phi_n^*=1+\frac{\epsilon}{1+\epsilon n m^{-1}}-\frac{\epsilon
m^{-1}}{1+\epsilon n m^{-1}}(d_0+d_1 z+d_2z^2+\ldots+d_nz^n)
\end{equation}
\[
=1+\frac{\epsilon}{1+2\epsilon }-\epsilon(b+bz^n+M_n+M^*_n)
\]
with
\[
b=\frac{1}{1+2\epsilon},\, M_n=\frac{ m^{-1}}{2(1+2\epsilon)}(d_1
z+d_2z^2+\ldots+d_{n-1}z^{n-1})
\]
The straightforward analysis shows that \eqref{choisi-7} implies
\eqref{condicion}. The formula \eqref{new-fr} differs from
\eqref{new1-fr}, in essence, only by the negative sign and the
normalization factor. Different sign makes checking conditions (1)
and (2) in the Lemma \ref{decop} harder when compared to the
argument in the proof of Theorem \ref{lbl}. However, in this
particular case, this can be done directly by analyzing the
polynomial $ d_0+d_1z+\ldots+d_nz^n $ around points $z=1$ and
$z=-1$. Indeed, we have
\[
\left|\sum_{j=1}^N\frac{\sin (j \theta )}{j}\right|<C
\]
uniformly over $\theta$ and $N$. Then
\eqref{taytay},\eqref{choisi-7}, and \eqref{leseq} imply  $ \Re
\Phi_n^*= 1+O(\epsilon),\, z\in \mathbb{T}$. Therefore,
\[
\int_{0}^{2\pi} |\Phi_n^*(e^{i\theta})|^{-2}d\theta<C_1
\]
and the opposite estimate
\[
\int_{0}^{2\pi} |\Phi_n^*(e^{i\theta})|^{-2}d\theta>C_2>0
\]
follows from the analysis of $\Phi_n^*$ away from $z=\pm 1$, i.e.,
on the arcs $z=e^{i\theta}, \epsilon<|\theta|<\pi-\epsilon$. Now, we
can normalize $\Phi_n^*$ and define $ \phi^*_n=a\Phi_n^* $ so that
\[
\int_0^{2\pi} |\phi_n^*(e^{i\theta})|^{-2}d\theta=2\pi
\]
For the constant $a$, we then have $a\sim 1$. Next, to check that
$\phi_n^*$ corresponds to a Steklov measure, one only needs to
modify the choice of $\widetilde F$ by changing the sign in front of
$\epsilon$:
\[
\widetilde F=1+C\epsilon M_n
\]
and repeating \eqref{ushi} with properly chosen $C$.\smallskip

 {\bf Remark.} As one can see from the proof of Theorem \ref{lbl}, the different
sort of cancelation has been used to show the Steklov condition of
the measure. In particular, the estimate \eqref{sec1} is violated as
$\phi_n^*-\phi_n=a(1-z^n)$ does not provide the strong cancelation
around $z=1$.

{\Large \part*{Acknowledgement.}} The research of S.D. was supported
by grants NSF-DMS-1464479 and  RSF-14-21-00025. The hospitality of
Keldysh Institute of Applied Mathematics in Moscow is gratefully
acknowledged.

\end{document}